\documentclass[14pt]{amsart}
\usepackage{cases}
\usepackage{amsmath}
\usepackage{amsfonts}
\usepackage{bm}
\usepackage{amsfonts,amsmath,amssymb,amscd,bbm,amsthm,mathrsfs,dsfont}
\usepackage{mathrsfs}
\usepackage{pb-diagram}
\usepackage{color}
\usepackage{amssymb}
\usepackage{xypic}
\usepackage{mathrsfs}
\usepackage{pifont}
\usepackage[all]{xy}
\usepackage{breqn}
\usepackage{graphicx}
\usepackage{epstopdf}
\usepackage{float}

\newtheorem{Theorem}{Theorem}[section]
\newtheorem{Lemma}[Theorem]{Lemma}
\newtheorem{Definition}[Theorem]{Definition}

\newtheorem{Proposition}[Theorem]{Proposition}

\newenvironment{Proof}[1][Proof]{\begin{trivlist}
\item[\hskip \labelsep {\bfseries #1}]}{\flushright
$\Box$\end{trivlist}}

\newcommand{\lra}{\longrightarrow}

\newcommand{\ra}{\rightarrow}
\newcommand{\sdp}{\times\kern-.2em\vrule height1.1ex depth-.05ex}
\newcommand{\epi}{\lra \kern-.8em\ra}
\newcommand{\A}{{\mathcal A}}

\newcommand{\Q}{{\mathbb Q}}

\newcommand{\Z}{{\mathbb Z}}

\newcommand{\Rnum}[1]{\uppercase\expandafter{\romannumeral #1\relax}}

 \normalbaselines
\begin{document}

\title {On inner Poisson structures of a quantum cluster algebra without coefficients}
\author{Fang Li}\author{Jie Pan$^*$}

\address{Fang Li
\newline Department of Mathematics, Zhejiang University (Yuquan Campus), Hangzhou, Zhejiang 310027, P. R. China}
\email{fangli@zju.edu.cn}

\address{Jie Pan
\newline Department of Mathematics, Zhejiang University (Yuquan Campus), Hangzhou, Zhejiang 310027, P. R. China}
\email{panjie_zhejiang@qq.com}

\thanks{*: Corresponding author}

\date{version of \today}

\maketitle
\bigskip
\begin{abstract}

The main aim of this article is to characterize inner Poisson structure on a quantum cluster algebra without coefficients. Mainly, we prove that inner Poisson structure on a quantum cluster algebra without coefficients is always a standard Poisson structure. In order to relate with compatible Poisson structure,  we introduce the concept of so-called locally inner Poisson structure on a quantum cluster algebra and then show it is equivalent to locally standard Poisson structure in the case without coefficients. Based on the result from \cite{LP} we obtain finally the equivalence between locally inner Poisson structure and compatible Poisson structure in this case.

\end{abstract}
\vspace{4mm}

\tableofcontents

\section{Introduction and preliminaries}
The introduction of quantum cluster algebras in \cite{BZ} is an important development of the theory of cluster algebras, which establishes a connection between cluster theory and the theory of quantum groups, see \cite{GLS} and \cite{GY}. And this is closely related to compatible Poisson structures on cluster algebras, see \cite{GSV}, \cite{N}. Moreover, in \cite{LP} we studied compatible Poisson structures on quantum cluster algebras and the theory of second quantization related to such Poisson structure.

In this article, we focus on another special kind of Poisson structures, i.e, inner Poisson structure on a quantum cluster algebra without coefficients. It is found  that an inner Poisson structure are in fact a standard Poisson structure and then it is connected to compatible Poisson structure via locally notion using the result in \cite{LP}.

We know from \cite{YYZ} that an inner Poisson structure on the path algebra of a finite connected quiver without oriented cycles is always piecewise standard. Together with our result, it shows that in some sense, on an algebra, non-commutativity conflicts with non-trivial inner Poisson structures. \\

First, we introduce some related notations and definitions.

For $n\leqslant m\in\mathbb{N}$, denote $\mathbb{T}_{n}$ the \textbf{$n$-regular tree} with vertices $t\in \mathbb{T}_{n}$.

\begin{Definition}
   (1)\;  A \textbf{quantum seed} at vertex $t\in\mathbb{T}_n$ is a triple $\Sigma=(\tilde{X}(t),\tilde{B}(t),\Lambda(t))$ such that\\
   $\bullet$\; $\tilde{B}(t)$ is an $m\times n$ integer matrix such that the principal part is skew-symmetrizable, i.e. there is a positive diagonal matrix $D$ satisfying $DB(t)$ is skew-symmetric, where $B(t)$ is the first $n$ rows of $\tilde{B}(t)$.\\
   $\bullet$\; $\Lambda(t)$ is an $m\times m$ skew-symmetric integer matrix and $(\tilde{B}(t),\Lambda(t))$ is a \textbf{compatible pair}, i.e,
   \begin{equation}\label{B^T Lambda=DO}
     \tilde{B}(t)^{\top}\Lambda(t)=\begin{pmatrix}
                                 D & O
                               \end{pmatrix}.
   \end{equation}
   $\bullet$\; The \textbf{(extended) cluster} $\tilde{X}(t)=(X_{t}^{e_{1}},X_{t}^{e_{2}},\cdots,X_{t}^{e_{n}},X^{e_{n+1}},\cdots,X^{e_{m}})$ at $t$ is an $m$-tuple satisfying
   \[X_{t}^{e}X_{t}^{f}=q^{\frac{1}{2}e\Lambda(t)f^{\top}}X_{t}^{e+f},\quad\forall e,f\in\Z^{m},\]
   where $\left\{e_{i}\right\}_{i=1}^{m}$ is the standard basis of $\mathbb{Z}^{m}$. $X_{t}^{e_{i}},i\in[1,n]$ are called \textbf{cluster variables} at $t$ while $X^{e_{i}},i\in[n+1,m]$ are called \textbf{frozen variables}.

      (2)\; For any $k\in[1,n]$, define the \textbf{mutation} $\mu_{k}$ at direction $k$ to be $\mu_{k}(\Sigma)=\Sigma^{\prime}=(\tilde{X}^{\prime},\tilde{B}^{\prime},\Lambda^{\prime})$ such that
        \begin{equation*}
          \mu_{k}(X_{t}^{e_{k}})=X_{t}^{-e_{k}+[b_{k}(t)]_{+}}+X_{t}^{-e_{k}+[-b_{k}(t)]_{+}}
        \end{equation*}
        where $b_{k}(t)$ is the $k$-th column of $\tilde{B}(t)$ and $[a]_{+}=max\left\{a,0\right\}$ pointwise for any vector $a$.
        \[\tilde{X}^{\prime}=(\tilde{X}(t))\backslash\left\{X_{t}^{e_{k}}\right\})\bigcup\left\{\mu_{k}(X_{t}^{e_{k}})\right\}.\]
        \[\tilde{B}^{\prime}=\mu_k(\tilde B(t))=(b_{ij}^{\prime})_{m\times n}\] satisfying that
        \begin{equation}\label{matrix mutation}
          b_{ij}^{\prime}=\left\{
          \begin{array}{lcr}
            -b_{ij}(t)&&if\quad i=k \quad or\quad j=k\\
            b_{ij}(t)+sgn(b_{ik}(t))[b_{ik}(t)b_{kj}(t)]_{+}&&otherwise
          \end{array}
          \right .
        \end{equation}
        And $\Lambda^{\prime}=\mu_k(\Lambda(t))=(\lambda_{ij}^{\prime})_{m\times m}$ satisfying
        \begin{equation}\label{lambdamutation}
          \lambda_{ij}^{\prime}=\left\{
          \begin{array}{lcr}
            -\lambda_{kj}(t)+\sum\limits_{l=1}^{m}[b_{lk}(t)]_{+}\lambda_{lj}(t)&&if\quad i=k\neq j\\
            -\lambda_{ik}(t)+\sum\limits_{l=1}^{m}[b_{lk}(t)]_{+}\lambda_{il}(t)&&if\quad j=k\neq i\\
            \lambda_{ij}(t)&&otherwise
          \end{array}
          \right .
        \end{equation}
\end{Definition}

Note that \ref{B^T Lambda=DO} requests $\tilde{B}(t)$ and $\Lambda(t)$ to be of full column rank. It can be verified that $\mu_{k}(\Sigma)$ is also a quantum seed and $\mu_{k}$ is an involution.

For the Laurent polynomial ring $\mathbb{Z}[q^{\pm\frac{1}{2}}]$ with a formal variable $q$, define the \textbf{quantum torus} $T_{t}$ at $t$ to be a $\mathbb{Z}[q^{\pm\frac{1}{2}}]$-algebra generated by $\tilde{X}(t)$. Denoted by $\mathscr{F}_{q}$ the skew-field of fractions of $T_{t}$. It does not depend on the choice of $t$.

\begin{Definition}
  Given seeds $\Sigma(t)=(\tilde{X}(t),\tilde{B}(t),\Lambda(t))$ at $t\in \mathbb{T}_{n}$ so that $\Sigma(t^{\prime})=\mu_{k}(\Sigma(t))$ for any $t-t^{\prime}$ in $\mathbb{T}_{n}$ connected by an edge labeled $k\in[1,n]$, then the $\mathbb{Z}[q^{\pm\frac{1}{2}}][X^{\pm e_{1}},\cdots,X^{\pm e_{m}}]$-subalgebra of $\mathscr{F}_{q}$ generated by all variables in $\bigcup\limits_{t\in\mathbb{T}_{n}}X(t)$ is called the \textbf{quantum cluster algebra} $A_{q}(\Sigma)$ (or simply $A_{q}$) associated with $\Sigma$.
\end{Definition}

A \textbf{Poisson structure} on an associative k-algebra $\A$ means a triple $(\A,\cdot,\left\{-,-\right\})$ where $(\A,\left\{-,-\right\})$ is a Lie k-algebra i.e. satisfying Jacobi identity such that the Leibniz rule holds: for any $a,b,c\in \A$,
\[\left\{a,bc\right\}=\left\{a,b\right\}c+b\left\{a,c\right\}.\]
Algebra $\A$ together with a Poisson structure on it is called a \textbf{Poisson algebra}. Denote the \textbf{Hamiltonian} of $a\in \A$ by
\[ham(a)=\left\{a,-\right\}\in End_{k}(\A,\A).\]
Then the Leibniz rule is equivalent to that $ham(a)$ is a derivation of $\A$ as an associative algebra for any $a\in \A$.
\begin{Definition}
  Let $\A$ be an associative algebra. $[a,b]=ab-ba$ is called the \textbf{commutator} of a and b, for any $a,b\in \A$. And for any $\lambda \in k$, $(\A,\cdot,\lambda[-,-])$ is a Poisson algebra called a \textbf{standard Poisson structure} on $(\A,\cdot)$.
\end{Definition}

A Poisson algebra $(\A,\cdot,\left\{-,-\right\})$ is said to be \textbf{inner} if $ham(a)=[a^{\prime},-]$ for some $a^{\prime}\in \A$, i.e. it is an inner derivation.\par
As a natural generalization of standard Poisson algebras, inner Poisson structures often arise, for examples: For an associative algebra $\A$,\\
(1)  If the first Hochschild cohomology of $(\A,\cdot)$ vanishes, then any Poisson structure on it is inner, see\cite{G}.\\
(2)  $\left(\right.$\cite{YYZ}$\left.\right)$
 Let $\left(\A,\cdot\right)$ be an associative algebras. Then a Poisson bracket $\left\{-,-\right\}$ on $\left(\A,\cdot\right)$ is an inner Poisson bracket if and only if there is  a $k$-linear transformation $g$ of $\A$ satisfying $ham(a)=[g(a),-]$ for any $a \in \A$ and
\begin{equation}\label{YYZ1}
  [g(x),y]=[x,g(y)], \forall x,y \in \A,
\end{equation}
\begin{equation}\label{YYZ2}
  [g(x),g(y)]-g([g(x),y]) \in Z(\A), \forall x,y \in \A.
\end{equation}

Furthermore, for any inner Poisson bracket $\left\{-,-\right\}$ on $\left(\A,\cdot\right)$, we can always find a $k$-linear transformation $g_0$ of $\A$ satisfying the above equations and meantime,
\begin{equation}\label{YYZ3}
  Z(\A)\subseteq Ker(g_0),
\end{equation}
where $Z(\A)$ is the center of the Lie bracket $[-,-]$.

Moreover, it is proved in \cite{YYZ} that for a finite connected quiver $Q$ without oriented cycles,
\[kQ=k\cdot 1\oplus \bigoplus \limits_{1\leqslant i\leqslant m}I_{i},\]
is a decomposition into indecomposable ideals of the Lie algebra $(kQ,[-,-])$. Furthermore, if $\left\{-,-\right\}$ is an inner Poisson structure on the path algebra $kQ$, then there is a unique vector $(\lambda_{1},\cdots,\lambda_{m})\in k^{m}$ such that
\begin{equation}\label{4}
  ham(a)=\lambda_{i}[a,-],\text{for any } a\in I_{i}, 1\leqslant i\leqslant m.
\end{equation}
Conversely, for any vector $(\lambda_{1},\cdots,\lambda_{m})\in k^{m}$, there is a unique inner Poisson structure on $kQ$ (up to a Poisson algebra isomorphism) satisfying (\ref{4}).

From now, let $\mathscr{P}(\A)$ be the set of the $k$-linear transformations of $A$ satisfying (\ref{YYZ1}), (\ref{YYZ2}). Define an equivalence relation $\sim$ on $\mathscr{P}(\A): g \sim g'$ if and only if there exists $\tau \in Aut(\A,\cdot)$ such that Im($\tau g \tau^{-1} - g')\subseteq Z(\A)$. Denote by [$g$] the equivalence class of $g$.

Two Poisson structures on $(\A,\cdot)$ are called {\bf isomorphic} as Poisson algebras if there exists an associative algebra automorphism $\tau$ of $(\A,\cdot)$ such that it is also a Lie algebra homomorphism. Denote by [$(\A,\cdot,\{-,-\})$] the iso-class of $(\A,\cdot,\{-,-\})$.

The paper is organized as follows.

In Section 2 we discuss the inner Poisson structures on a quantum cluster algebra without coefficients and prove the main theorem.
\begin{Theorem}
(Theorem \ref{r1})\;Let $A_{q}$ be a quantum cluster algebra without coefficients, any inner Poisson structure on $A_{q}$ must be a standard Poison structure.
\end{Theorem}

Then we generalize the definition to locally inner Poisson structures and find following equivalence.

\begin{Theorem}
  (Theorem \ref{r3})\;Let $A_{q}$ be a quantum cluster algebra without coefficients and $\left\{-,-\right\}$ a Poisson structure on $A_{q}$. The following statements are equivalent:

  (1)\; $\left\{-,-\right\}$ is locally standard.

  (2)\; $\left\{-,-\right\}$ is locally inner.

  (3)\; $\left\{-,-\right\}$ is compatible with $A_{q}$.
\end{Theorem}

\vspace{4mm}
\section{Proof of the main theorem}
The following theorem from \cite{YYZ} gives a correspondence between inner Poisson brackets and k-linear transformations.
\begin{Theorem}[\cite{YYZ}]\label{g}
  Let $\left(\A,\cdot\right)$ be an associative algebras. Then the map
  \begin{equation*}
  \left \{\text{the equivalence classes of }\mathscr{P}(\A) \right \} \rightarrow \left \{\text{the isoclasses of inner Poisson structures on }(\A,\cdot) \right \}
  \end{equation*}
  given by
  \begin{equation*}
  [g] \mapsto [\left(\A,\cdot,\{-,-\}\right)],\; \text{where}\; ham(a)=[g(a),-], \forall a \in \A
  \end{equation*}
  is bijective.
\end{Theorem}
Because of the above theorem, we can focus on the $k$-linear transformations when studying inner Poisson structures. In this section, we study the inner Poisson structures of a quantum cluster algebra $A_{q}$ with deformation matrix $\Lambda$.

Because $g$ is $k$-linear, we only need to think about its action on Laurent monomials in $A_{q}$. In this section when we say Laurent monomials, we actually mean Laurent monomials in the initial cluster.
\begin{Lemma}\label{inner} For a quantum cluster algebra $A_q$, if  $g \in \mathscr{P}(A_{q})$, then for any $h\in[1,m]$ and any cluster $\tilde{X}=\{X_1,\cdots,X_n,X_{n+1},\cdots,X_m\}$,  we have
   \[g(X_{h})=k_{1}^{h}X_{h}+\sum \limits_{i=2}^{l_{h}}k_{i}^{h}X_{1}^{a_{i1}^{h}}X_{2}^{a_{i2}^{h}}\cdots X_{m}^{a_{im}^{h}},\]
    which is expanded in a $\mathbb{Z}[q^{\pm \frac{1}{2}}]$-linearly independent form, with $l_{h}\in \mathbb{N}$, $a_{i1}^{h},\cdots,a_{im}^{h}\in \mathbb{Z}$ and $k_1^h, k_{i}^{h} \in \mathbb{Z}[q^{\pm \frac{1}{2}}]$ for $2\leqslant i \leqslant l_{h}$, satisfying that
    \begin{equation}\label{solution}
    (a_{i1}^{h},a_{i2}^{h},\ldots,a_{im}^{h})\Lambda =(\lambda_{h1}^{i},\lambda_{h2}^{i},\ldots,\lambda_{hm}^{i})
    \end{equation}
    where $\lambda_{hp}^{i}= 0$ or $\lambda_{hp}$ for $1\leq p\leq m$.
\end{Lemma}

\begin{Proof}
  Assume $g(X_{1})=\sum \limits_{i=1}^{l_{1}}k_{i}X_{1}^{a_{i1}}X_{2}^{a_{i2}}\cdots X_{m}^{a_{im}},k_{i}\neq 0$ and  $g(X_{2})=\sum \limits_{i=1}^{l_{2}}p_{i}X_{1}^{b_{i1}}X_{2}^{b_{i2}}\cdots X_{m}^{b_{im}},p_{i}\neq 0$,  are expanded in  $\mathbb{Z}[q^{\pm \frac{1}{2}}]$-linearly independent  forms. And assume $\left\{-,-\right\}$ is the inner Poisson bracket corresponding to $g$, i.e. $\left\{X,Y\right\}=[g(X),Y]$ for any $X,Y\in A_{q}$.

  Because $0=\{X_{1},X_{1}\}=[g(X_{1}),X_{1}]=\sum \limits_{i=1}^{l_{1}}(q^{\sum \limits_{t=1}^{m} a_{it}\lambda_{t1}}-1)X_{1}^{a_{i1}+1}X_{2}^{a_{i2}}\cdots X_{m}^{a_{im}}$, we have
  \[\sum \limits_{t=1}^{m} a_{it}\lambda_{t1}= 0 \;\;\; \text{for any}\;\;\; 1< i< l_1.\]
  Similarly,
   \[ \sum \limits_{t=1}^{m} b_{it}\lambda_{t2}= 0 \;\;\; \text{for any}\;\;\; 1< i< l_2.\]

   Moreover, according to (\ref{YYZ1}), we then obtain that
  \begin{equation}\label{expand1}
  \{X_{1},X_{2}\}=[g(X_{1}),X_{2}]=\sum \limits_{i=1}^{l_1}k_{i}(q^{\sum \limits_{t=2}^{m}a_{it}\lambda_{t2}}-q^{a_{i1}\lambda_{21}})X_{1}^{a_{i1}}X_{2}^{a_{i2}+1}X_{3}^{a_{i3}}\cdots X_{m}^{a_{im}};
  \end{equation}
  \begin{equation}\label{expand2}
  \{X_{1},X_{2}\}=[X_{1},g(X_{2})]=\sum \limits_{i=1}^{l_2}p_{i}(1-q^{\sum \limits_{t=1}^{m}b_{it}\lambda_{t1}})X_{1}^{b_{i1}+1}X_{2}^{b_{i2}}X_{3}^{b_{i3}}\cdots X_{m}^{b_{im}}.
  \end{equation}

 Trivially,  the expansions of the right-sides of (\ref{expand1}) and (\ref{expand2}) are also in $\mathbb{Z}[q^{\pm \frac{1}{2}}]$-linearly independent forms,  which are the same due to the algebraic independence of $\{X_1, X_2, \cdots, X_m\}$. Hence there exists $l_0\leq l_1,l_2$ such that there are $l_0$ monomials with non-zero-coefficients in the expansions of the right-sides of (\ref{expand1}) and (\ref{expand2}) respectively and the coefficients of other monomials are all zeros.

 Without loss of generality, suppose these $l_0$ monomials with non-zero-coefficients are just the first $l_0$ ones in the expansions of the right-sides of (\ref{expand1}) and (\ref{expand2}) respectively. We may assume they are in one-by-one correspondence indexed by $i=1,2,\cdots,l_0$. Hence due to the above discussion, we obtain that

  Case 1: For $i$ satisfying $1\leqslant i\leqslant l_{0}$,
  \begin{equation*}
    \left \{
    \begin{array}{l}
        a_{i1}=b_{i1}+1 \\
      a_{i2}+1=b_{i2} \\
      a_{it}=b_{it}, \; \text{for}\; 3\leqslant t\leqslant m\\
          k_{i}(q^{\sum \limits_{t=2}^{m}a_{it}\lambda_{t2}}-q^{a_{i1}\lambda_{21}})= p_{i}(1-q^{\sum \limits_{t=1}^{m}b_{it}\lambda_{t1}})\neq 0\\
      \sum \limits_{t=1}^{m} a_{it}\lambda_{t1}= \sum \limits_{t=1}^{m} b_{it}\lambda_{t2}= 0
    \end{array}
    \right .
  \end{equation*}

  Case 2: For $i,j$ satisfying $l_{0}<i\leqslant l_{1}$, $l_{0}<j\leqslant l_{2}$, we have
    \[q^{\sum \limits_{t=2}^{m}a_{it}\lambda_{t2}}-q^{a_{i1}\lambda_{21}}=1-q^{\sum \limits_{t=1}^{m}b_{jt}\lambda_{t1}}=\sum \limits_{t=1}^{m} a_{it}\lambda_{t1}= \sum \limits_{t=1}^{m} b_{jt}\lambda_{t1}= 0.\]

 From Case 1,  we get that for $1\leq i\leq l_0$,
    \begin{equation}\label{case1}
    \left \{
    \begin{array}{l}
      \sum \limits_{t=1}^{m} a_{it}\lambda_{t1}= 0 \\
      \sum \limits_{t=1}^{m} a_{it}\lambda_{t2}=\sum \limits_{t=1}^{m} b_{it}\lambda_{t2}+ \lambda_{12}-\lambda_{22}=\lambda_{12}.

    \end{array}
    \right .
  \end{equation}

  From Case 2, we have that for $l_0<i\leq l_1$,
  \begin{equation}\label{case2}
    \left \{
    \begin{array}{l}
      \sum \limits_{t=1}^{m} a_{it}\lambda_{t1}= 0 \\
      \sum \limits_{t=1}^{m} a_{it}\lambda_{t2}= 0.
    \end{array}
    \right .
  \end{equation}

      In the above discussion, replacing $X_{2}$ by other $X_{p}$ for $p\not=1,2$, we get similarly that:
    \begin{equation}\label{casep}
    \sum \limits_{t=1}^{m} a_{it}\lambda_{tp}= \lambda_{1p}^{i}
     \end{equation}
     where $\lambda_{1p}^{i}=\lambda_{1p}$ or 0 for any $3\leq p\leq m, 1\leq i\leq l_1$.

  In summary from (\ref{case1}), (\ref{case2}) and (\ref{casep}), we have
       $(a_{i1},a_{i2},\cdots,a_{im})\Lambda =(\lambda_{11}^{i},\lambda_{12}^{i},\cdots,\lambda_{1m}^{i})$
  for any $ 1\leq i\leq l_1$, i.e, in the expansion of $g(X_{1})$ any term $k_{i}X_{1}^{a_{i1}}X_{2}^{a_{i2}}\cdots X_{m}^{a_{im}}$ with $k_{i}\neq 0$ must have $(a_{i1},\cdots,a_{im})$ as a solution of above equations.

 When $(a_{11},a_{12},\cdots,a_{1m})=(1,0,\cdots,0)$, (\ref{solution}) is satisfied for $\lambda_{1p}^{1}=\lambda_{1p}$ for any $p$. So in the expansion of $g(X_1)$, we may consider the monomial $k_1X_1$ as the first term, i.e. $i=1$. Note that it maybe not exist if its coefficient $k_1$ is zero.

 Then we have the expansion of $g(X_1)$ as follows:
  \[g(X_{1})=k_{1}X_{1}+\sum \limits_{i=2}^{l_{1}}k_{i}X_{1}^{a_{i1}}X_{2}^{a_{i2}}\cdots X_{m}^{a_{im}},\]
  and
  \[(a_{i1},a_{i2},\ldots,a_{im})\Lambda =(\lambda_{11}^{i},\lambda_{12}^{i},\ldots,\lambda_{1m}^{i})\]
  where $\lambda_{1k}^{i}= 0$ or $\lambda_{1k}$. It implies this lemma holds for $h=1$.

The similar discussion for any $X_{h}, h\in [1,m]$ can be given to complete the proof. \end{Proof}

In the rest of this section we will always assume $A_{q}$ is a quantum cluster algebra without coefficients, i.e, $m=n$. Then (\ref{B^T Lambda=DO}) becomes
\[B^{\top}\Lambda=D\]
Following this, $B$ and $\Lambda$ are both of rank n and invertible. So $n>1$ since $B=0$ when $n=1$. And in this case $(a_{i1},a_{i2},\ldots,a_{im}) =(\lambda_{11}^{i},\lambda_{12}^{i},\ldots,\lambda_{1m}^{i})\Lambda^{-1}$.\begin{Lemma}\label{I1}
  Let $A_{q}$ be a quantum cluster algebra without coefficients. If $g \in \mathscr{P}(A_{q})$ satisfies that $g(X)=k_{X}X$ for any Laurent monomial $X$ in $A_{q}$ with $k_{X}\in \mathbb{Z}[q^{\pm\frac{1}{2}}]$, then there is a scalar $\mathbb{Z}[q^{\pm\frac{1}{2}}]$-transformation  $g^{\prime}\in \mathscr{P}(A_{q})$ such that $g^{\prime}\sim g$.
\end{Lemma}
\begin{Proof}{}

    For any  Laurent monomial $X=pX_{1}^{m_{1}}X_{2}^{m_{2}}\cdots X_{n}^{m_{n}} \in A_{q}$, $X$ communicates with $X_{i}$ if and only if $\sum \limits_{j} m_{j}\lambda_{ji}$ = 0. Therefore $X\in Z(A_q)$ the center of $A_q$ if and only if $(m_{1},m_{2},\cdots,m_{n})\Lambda$ = 0. Because $\Lambda$ is invertible, we have $Z(A_{q})= \mathbb{Z}[q^{\pm \frac{1}{2}}]$.

    Therefore for any non-constant Laurent monomial $X\in A_{q}$, we can find a Laurent monomial $Y_0\in A_{q}$ such that $\left [X,Y_0\right ]\neq$ 0.

    For two non-constant Laurent monomials $X,Y\in A_{q}$, we claim $k_X=k_Y$.

    Case 1: Assume $XY\not=YX$.

   Denote the Poisson bracket associated to $g$ as $\left\{-,-\right\}$. Then, first,  we have $$\left\{X,Y\right\}=[g(X),Y]=k_{X}XY-k_{X}YX.$$
    On the other hand, according to (\ref{YYZ1}) we have also $$\left\{X,Y\right\}=[X,g(Y)]=k_{Y}XY-k_{Y}YX.$$ Thus since  $XY\neq YX$, we obtain $k_{X}=k_{Y}$.

    Case 2: Assume $XY=YX$.

    Since $X,Y\notin Z(A_{q})$, there are Laurent monomials $M,N$ in $A_{q}$ such that $XM\neq MX, YN\neq NY$.
   Then from Case 1, we have $k_X=k_M$, $k_Y=k_N$.

    If either $YM\neq MY$ or $XN\neq NX$, then $k_{Y}=k_{M}$ or $k_{X}=k_{N}$. It follows that $k_X=k_Y$.

    Otherwise, $YM=MY$ and $XN=NX$. It is easy to see that  $X(MN)\neq (MN)X, Y(MN)\neq (MN)Y$. So, from Case 1,  $k_{X}=k_{MN}=k_{Y}$.

Then, there exists a fixed element $k_0\in\mathbb{Z}[q^{\pm\frac{1}{2}}]$ such that $k_0=k_X$   for any non-constant Laurent monomial $X \in A_{q}$. It follows that for any such  $X$,
\begin{equation}\label{scalar1}
 g(X)=k_0X.
 \end{equation}

  For any constant $a\in \mathbb{Z}[q^{\pm\frac{1}{2}}]$ and any $W\in A_{q}$, we have
  \[[g(a),W]=\{a,W\}=[a,g(W)]=0.\]
  Therefore $g(a)\in Z(A_{q})= \mathbb{Z}[q^{\pm \frac{1}{2}}]$, that is, $g(a)$ is a constant.

Let $g'$ be the $k_0$-scalar linear transformation of $A_q$, that is, for any $W\in A_q$, define
$g'(W)=k_0W$.
Trivially, $g^{\prime}\in \mathscr{P}(A_{q})$.

  By (\ref{scalar1}) and since $g(a)$ is a constant for any $a\in \mathbb{Z}[q^{\pm\frac{1}{2}}]$, we have  Im$(g-g^{\prime})\subseteq Z(A_q)= \mathbb{Z}[q^{\pm \frac{1}{2}}]$. It means that $g\sim g^{\prime}$.

\end{Proof}

\begin{Lemma}\label{I2}
  Let $A_{q}$ be a quantum cluster algebra without coefficients. Then for any $g \in \mathscr{P}(A_{q})$,

   (i)\; for any Laurent monomial $X$ in $A_q$, $g(X)=k_{X}X+k_{X}^{\prime}$, where $k_{X},k_{X}^{\prime}\in \mathbb{Z}[q^{\pm \frac{1}{2}}]$;

   (ii)\;  there is a scalar $\mathbb{Z}[q^{\pm\frac{1}{2}}]$-transformation  $g_{0}\in \mathscr{P}(A_{q})$ such that $g_{0}\sim g$.
\end{Lemma}
\begin{Proof} {\em (i)}\;
  According to Lemma \ref{inner},
  \begin{equation}\label{equalfromg}
  g(X_{h})=k_{1}^{h}X_{h}+\sum \limits_{i=2}^{l_{h}}k_{i}^{h}X_{1}^{a_{i1}^{h}}X_{2}^{a_{i2}^{h}}\cdots X_{n}^{a_{in}^{h}},
  \end{equation}
  where
  \begin{equation*}
    (a_{i1}^{h},a_{i2}^{h},\ldots,a_{in}^{h})\Lambda =(\lambda_{h1}^{i},\lambda_{h2}^{i},\ldots,\lambda_{hn}^{i})
  \end{equation*}
  and $\lambda_{hp}^{i}= 0$ or $\lambda_{hp}$ for $1\leq p\leq n$.
  For $m_{1},\cdots,m_{n}\in\mathbb{Z}$, assume
  \begin{equation}\label{expen}
  g(X_{1}^{m_{1}}X_{2}^{m_{2}}\cdots X_{n}^{m_{n}})=\sum \limits_{j=1}^{l}f_{j}X_{1}^{c_{j1}}X_{2}^{c_{j2}}\cdots X_{n}^{c_{jn}},
   \end{equation}
   satisfying $c_{1t}=m_{t}$ for $t\in[1,n]$,  as  a $\mathbb{Z}[q^{\pm \frac{1}{2}}]$-linearly independent expansion except that  $f_{1}$ may be zero. Let $\left\{-,-\right\}$ be the Poisson structure correspond to $g$.

  According to (\ref{YYZ1}), we have:
  \begin{equation*}
    \begin{array}{rl}
      \left \{X_{1},X_{1}^{m_{1}}X_{2}^{m_{2}}\cdots X_{n}^{m_{n}}\right \}= & [g(X_{1}),X_{1}^{m_{1}}X_{2}^{m_{2}}\cdots X_{n}^{m_{n}}] \\
      = & k_{1}^{1}X_{1}^{m_{1}+1}X_{2}^{m_{2}}\cdots X_{n}^{m_{n}}- k_{1}^{1}X_{1}^{m_{1}}X_{2}^{m_{2}}\cdots X_{n}^{m_{n}}X_{1} \\
        & +\sum \limits_{i=2}^{l_{1}}k_{i}^{1}X_{1}^{a_{i1^{1}}}\cdots X_{n}^{a_{in}^{1}}X_{1}^{m_{1}}\cdots X_{n}^{m_{n}}- \sum\limits_{i=2}^{l_{1}}k_{i}^{1}X_{1}^{m_{1}}\cdots X_{n}^{m_{n}}X_{1}^{a_{i1}^{1}}\cdots X_{n}^{a_{in}^{1}} \\
      = & k_{1}^{1}(1-q^{\sum \limits_{t=1}^{n}m_{t}\lambda_{t1}})X_{1}^{m_{1}+1}X_{2}^{m_{2}}\cdots X_{n}^{m_{n}} \\
        & + \sum \limits_{i=2}^{l_{1}}k_{i}^{1}(q^{\sum \limits_{r>s}a_{ir}m_{s}\lambda_{rs}}-q^{\sum \limits_{r<s}a_{ir}m_{s}\lambda_{sr}})X_{1}^{m_{1}+a_{i1}^{1}}\cdots X_{n}^{m_{n}+a_{in}^{1}} \\
      = & k_{1}^{1}(1-q^{\sum \limits_{t=1}^{n}m_{t}\lambda_{t1}})X_{1}^{m_{1}+1}X_{2}^{m_{2}}\cdots X_{n}^{m_{n}} \\
        & + \sum \limits_{i=2}^{l_{1}}k_{i}^{1}q^{\sum \limits_{r>s}a_{ir}^{1}m_{s}\lambda_{rs}}(1-q^{\sum \limits_{r,s=1}^{n}a_{ir}^{1}m_{s}\lambda_{sr}})X_{1}^{m_{1}+a_{i1}^{1}}\cdots X_{n}^{m_{n}+a_{in}^{1}};
    \end{array}
  \end{equation*}
  on the other hand,
  \begin{equation*}
    \begin{array}{rl}
      \left \{X_{1},X_{1}^{m_{1}}X_{2}^{m_{2}}\cdots X_{n}^{m_{n}}\right \}= & [X_{1},g(X_{1}^{m_{1}}X_{2}^{m_{2}}\cdots X_{n}^{m_{n}})] \\
      = & \sum \limits_{j=1}^{l}f_{j}(X_{1}^{c_{j1}+1}X_{2}^{c_{j2}}\cdots X_{n}^{c_{jn}}- X_{1}^{c_{j1}}X_{2}^{c_{j2}}\cdots X_{n}^{c_{jn}X_{1}}) \\
      = & \sum \limits_{j=1}^{l}f_{j}(1-q^{\sum \limits_{t=1}^{n}c_{jt}\lambda_{t1}})X_{1}^{c_{j1}+1}X_{2}^{c_{j2}}\cdots X_{n}^{c_{jn}}
    \end{array}
  \end{equation*}
  Note that in the last step of the first expansion of $\{X_{1},X_{1}^{m_{1}}X_{2}^{m_{2}}\cdots X_{n}^{m_{n}} \}$, we have
  \begin{equation}\label{oneequality}
  \sum \limits_{r,s=1}^{n}a_{ir}^{1}m_{s}\lambda_{sr}= -\sum \limits_{s=1}^{n}(\sum \limits_{r=1}^{n}a_{ir}^{1}\lambda_{rs})m_{s}= -(a_{i1}^{1}\cdots a_{in}^{1})\Lambda (m_{1}\cdots m_{n})^{\top}=-(\lambda_{11}^{i}\cdots \lambda_{1n}^{i})(m_{1}\cdots m_{n})^{\top}.
  \end{equation}

  The last steps of the two kinds of expansions of $\{X_{1},X_{1}^{m_{1}}X_{2}^{m_{2}}\cdots X_{n}^{m_{n}} \}$ are both in $\mathbb{Z}[q^{\pm \frac{1}{2}}]$-linearly independent forms, which are the same due to the algebraic independence of $\{X_1, X_2, \cdots, X_m\}$. Hence, for some $l_0\leq l_1,l$,  there are $l_{0}-1$ monomials with non-zero-coefficients in the last steps of two kinds of expansions above  respectively and the coefficients of other monomials are all zeros, besides the first terms in these two expansion which maybe be zero or non-zero in the various cases.

  Without loss of generality, suppose the $l_{0}-1$ monomials with non-zero-coefficients are just those ones whose indexes are with $2\leqslant i\leqslant l_{0}$ and $2\leqslant j\leqslant l_{0}$ respectively in the last steps of two kinds of expansions above, that is, we assume they are in one-by-one correspondence indexed by $i=2,\cdots,l_0$.

 Thus, due to the above discussion, from comparation of coefficients,  we obtain that
   \begin{equation}\label{i=1}
  A_1=k_{1}^{1}(1-q^{\sum \limits_{t=1}^{n}m_{t}\lambda_{t1}})X_{1}^{m_{1}+1}X_{2}^{m_{2}}\cdots X_{n}^{m_{n}}= f_{1}(1-q^{\sum \limits_{t=1}^{n}c_{1t}\lambda_{t1}})X_{1}^{c_{11}+1}X_{2}^{c_{12}}\cdots X_{n}^{c_{1n}}.
   \end{equation}
    When $2\leqslant i\leqslant l_{0}$,
     \begin{equation}\label{i=2}
     0\neq k_{i}^{1}q^{\sum \limits_{r>s}a_{ir}^{1}m_{s}\lambda_{rs}}(1-q^{\sum \limits_{r,s=1}^{n}a_{ir}^{1}m_{s}\lambda_{sr}})X_{1}^{m_{1}+a_{i1}^{1}}\cdots X_{n}^{m_{n}+a_{in}^{1}}= f_{i}(1-q^{\sum \limits_{t=1}^{n}c_{it}\lambda_{t1}})X_{1}^{c_{i1}+1}X_{2}^{c_{i2}}\cdots X_{n}^{c_{in}}.
     \end{equation}
  When $l_{0}<i\leqslant l_{h}, l_{0}<j\leqslant l$,
   \begin{equation}\label{i>l_0}
   0=k_{i}^{1}q^{\sum \limits_{r>s}a_{ir}^{1}m_{s}\lambda_{rs}}(1-q^{\sum \limits_{r,s=1}^{n}a_{ir}^{1}m_{s}\lambda_{sr}})X_{1}^{m_{1}+a_{i1}^{1}}\cdots X_{n}^{m_{n}+a_{in}^{1}}= f_{j}(1-q^{\sum \limits_{t=1}^{n}c_{jt}\lambda_{t1}})X_{1}^{c_{j1}+1}X_{2}^{c_{j2}}\cdots X_{n}^{c_{jn}}.
\end{equation}

  In (\ref{i=1}), we have that $A_1=0$ if and only if $\sum \limits_{t=1}^{n}m_{t}\lambda_{t1}= 0$; otherwise, $A_1\not=0$ then $f_{1}=k_{1}^{1}$.

  From (\ref{i=2}) and (\ref{oneequality}), we obtain that for $2\leqslant i\leqslant l_{0}$,
  \begin{equation*}
    \left \{
    \begin{array}{l}
      c_{i1}=m_{1}+a_{i1}^{1}-1\\
      c_{ip}=m_{p}+a_{ip}^{1},\; \text{for}\; 2\leqslant p \leqslant n\\
      0\neq \sum \limits_{r,s=1}^{n} a_{ir}m_{s}\lambda_{sr}= -\sum \limits_{s=1}^{n}\lambda_{1s}^{i}m_{s} \\
      0\neq \sum \limits_{t=1}^{n} c_{it}\lambda_{t1}= \sum \limits_{t=1}^{n}(m_{t}+a_{it}^{1})\lambda_{t1}= \sum \limits_{t=1}^{n}m_{t}\lambda_{t1}+\lambda_{11}^{i}= \sum \limits_{t=1}^{n}m_{t}\lambda_{t1}.
    \end{array}
    \right .
  \end{equation*}

  From (\ref{i>l_0}) and (\ref{oneequality}), we obtain that for $i,j>l_{0}$,
  \begin{equation}\label{added}
  \sum \limits_{s=1}^{n}\lambda_{1s}^{i}m_{s}=\sum \limits_{p=1}^{n}c_{jp}\lambda_{p1}=0.
  \end{equation}

  In conclusion, $(c_{j1},\cdots,c_{jn})$ with $f_{j}\neq 0$ must satisfy one of (\ref{i=1}), (\ref{i=2}) and (\ref{i>l_0}) for any $j=1,2,\cdots,l_0$.

  In the same way,  replacing $X_{1}$ by $X_{h}$, $h\in[1,n]$, we will also obtain three equalities similar to (\ref{i=1}), (\ref{i=2}) and (\ref{i>l_0}) such that $(c_{j1},\cdots,c_{jn})$ with $f_{j}\neq 0$ satisfies one of three equalities.

  According to our assumption, we always have $(c_{11},\cdots,c_{1n})=(m_{1},\cdots,m_{n})$.

  Now we want to prove by contradiction that  $(c_{j1},\cdots,c_{jn})$ can only be $(0,\cdots,0)$  for $2\leqslant j\leqslant l$.  Hence, we first assume that $(c_{j1},\cdots,c_{jn})\neq (0,\cdots,0)$ in this case.

  We can choose some special $m^o_{1},\cdots,m^o_{n}\in\mathbb{Z}_{\geqslant 0}^{n}$ such that
  \begin{equation}\label{m}
    \left \{
    \begin{array}{l}
      \sum \limits_{t=1}^{n}m^o_{t}\lambda_{th}\neq \lambda_{ih} \text{ for any }i,h\\
      \sum \limits_{t=1}^{n}m^o_{t}\lambda_{th}\neq 0 \text{ for any }i,h\\
      m^o_{t}\geqslant 0\text{ for any }t
    \end{array}
    \right .
  \end{equation}

  For any $X_{h}$, $h\in[1,n]$, we first claim that under the condition (\ref{m}), $(c_{j1},\cdots,c_{jn})$ with $f_{j}\neq 0$ does not satisfy the equality similar to (\ref{i=2}).

  In fact, because $\Lambda$ is invertible, so since $(c_{j1},\cdots,c_{jn})\neq (0,\cdots,0)$, we have  $(c_{j1},\cdots,c_{jn})\Lambda \neq (0,\cdots,0)$. Therefore $(c_{j1},\cdots,c_{jn})$ can not satisfy an equality similar to (\ref{i>l_0}) for all $h\in[1,n]$, i.e. it must satisfy some equations similar to (\ref{i=1}) or (\ref{i=2}) for some $h$. Therefore all of the possible $(c_{j1},\cdots,c_{jn})\not=(0,\cdots,0)$ are  $(m^o_{1}+a_{i1}^{h},\cdots,m^o_{h}+a_{ih}^{h}-1,\cdots,m^o_{n}+a_{in}^{h})$ for some $i$ and $h$. Hence  for any $h$, by (\ref{m}),
  $$\sum \limits_{t=1}^{n}c_{jt}\lambda_{th}=\sum \limits_{t=1}^{n}(m^o_{t}+a_{it}^{p})\lambda_{th}-\lambda_{ph}= \sum \limits_{t=1}^{n}m^o_{t}\lambda_{th}+\lambda_{ph}^{i}-\lambda_{ph}= \left\{\begin{array}{cc}
                                                                      \sum \limits_{t=1}^{n}m^o_{t}\lambda_{th}\neq 0, & \text{if}\;\;\lambda_{ph}^{i}=\lambda_{ph} \\
                                                                      \sum \limits_{t=1}^{n}m^o_{t}\lambda_{th}-\lambda_{ph}\neq 0, & \text{if}\;\;\lambda_{ph}^{i}=0.
                                                                    \end{array}\right.$$
  So for any $h$, $(0,\cdots,0)$ is the only $(c_{j1},\cdots,c_{jn})$ satisfying the equalities  similar to (\ref{i>l_0}), $(m^o_{1},\cdots,m^o_{n})$ is the only $(c_{11},\cdots,c_{1n})$ satisfying the equalities similar to (\ref{i=1}), while all of that $$(m^o_{1}+a_{i1}^{p},\cdots,m^o_{p}+a_{ip}^{p}-1,\cdots,m^o_{n}+a_{in}^{p})$$ satisfy the equalities similar to (\ref{i=2}).

  Hence for any $(c_{j1},\cdots,c_{jn})\not=(0,\cdots,0)$ with $2\leqslant j\leqslant l_0$, we have
  \begin{alignat}{2}\label{not case 2}
       c_{j1}&=m^o_{1}+a_{j1}^{1}-1&=\cdots&=m^o_{1}+a_{j1}^{n}\nonumber\\
       \vdots\quad&\qquad \qquad\vdots&\quad \ddots&\qquad \qquad \vdots\\
       c_{jn}&=m^o_{n}+a_{jn}^{1}&=\cdots&=m^o_{1}+a_{jn}^{n}-1\nonumber.
  \end{alignat}
  Then for any $h_{1}\neq h_{2}$,
  \begin{equation*}
    \left \{
    \begin{array}{l}
       a_{jh_{1}}^{h_{1}}-1=a_{jh_{1}}^{h_{2}} \\
       a_{jh_{2}}^{h_{1}}=a_{jh_{2}}^{h_{2}}-1\\
       a_{jh}^{h_{1}}=a_{jh}^{h_{2}},\quad \text{for any } h\neq h_{1},h_{2}
    \end{array}
    \right .
  \end{equation*}
  Therefore,
  \[\sum \limits_{h=1}^{n}a_{jh}^{h_{1}}\lambda_{hh_{1}}=\lambda_{h_{1}h_{1}}^{j}=0=\lambda_{h_{1}h_{1}}, \;\; \sum\limits_{h=1}^{n}a_{jh}^{h_{1}}\lambda_{hh_{2}}=\sum \limits_{h=1}^{n}a_{jh}^{h_{2}}\lambda_{hh_{2}}+\lambda_{h_{1}h_{2}}=\lambda_{h_{2}h_{2}}^{j}+\lambda_{h_{1}h_{2}}=\lambda_{h_{1}h_{2}},\]
  for any $h_{1},h_{2}$. Hence $ (a_{j1}^{h_{1}},\cdots,a_{jn}^{h_{1}})\Lambda=(\lambda_{h_{1}1},\cdots,\lambda_{h_{1}n})$. Again because $\Lambda$ is invertible, we have $(a_{j1}^{h_{1}},\cdots,a_{jn}^{h_{1}})=e_{h_{1}}$. Therefore by (\ref{not case 2}) we get \[(c_{j1},\cdots,c_{jn})=(m^o_{1},\cdots,m^o_{n})+(a_{j1}^{h_{1}},\cdots,a_{jn}^{h_{1}})-e_{h_{1}}=(m^o_{1},\cdots,m^o_{n}),\]
  which contradicts to our assumption as $j\geqslant 2$. Thus in conclusion, under the condition (\ref{m}), $(c_{j1},\cdots,c_{jn})$ with $f_{j}\neq 0$ does not satisfy the equality similar to (\ref{i=2}) for any $h$.

  Hence, indeed, the case satisfying the equality similar to (\ref{i=2}) would not happen under the assumption of (\ref{m}). It means for any $h$, we only have (\ref{i>l_0}) to hold for $2\leqslant i\leqslant l_{h}$. Therefore, $\sum \limits_{s=1}^{n}\lambda_{hs}^{i}m^o_{s}=0$ for any $h$ and $2\leqslant i\leqslant l_{h}$ according to (\ref{added}).
  Define sets£º

  $S_{ih}=\{(m_{1},\cdots,m_{n})\in \mathbb Z^n_{\geqslant 0}\mid\; (m_{1},\cdots,m_{n})\Lambda=(t_{1},\cdots,t_{h-1},\lambda_{ih},t_{h+1},\cdots,t_{n}) \;\forall t_{1},\cdots,t_{h-1},t_{h+1},\\\cdots,t_{n}\in\mathbb{Z} \}$ for any $i,h\in [1,n]$;

  $T_h=\{(m_{1},\cdots,m_{n})\in \mathbb Z^n_{\geqslant 0}\mid\; (m_{1},\cdots,m_{n})\Lambda=(t_{1},\cdots,t_{h-1},0,$ $t_{h+1},\cdots,t_{n})\;\forall t_{1},\cdots,t_{h-1},t_{h+1},\cdots,\\t_{n}\in \mathbb Z\}$ for any $h\in [1,n]$.

  Then the set of positive integer vectors $(m_{1},\cdots,m_{n})$ satisfying (\ref{m}) is equal to the set $$\Z_{\geqslant 0}^{n}\backslash \bigcup_{i,h\in [1,n]}(S_{ih}\cup T_i).$$

  For any $i,h$, the sets $S_{ih}, T_i$ lie discretely in their corresponding $(n-1)$-dimensional nonnegative cones $C_{ih}, D_i$ in  $\Q^{n}$ respectively. All of $C_{ih}, D_i$ are contained  in the $n$-dimensional nonnegative cone (or say, the first quadrant) $\Q_{\geqslant 0}^{n}$ of  $\Q^{n}$. Let $C'= \bigcup\limits_{i,h\in [1,n]}(C_{ih}\cup D_i)$.

  It is easy to see that every $l$-dimensional nonnegative cone included in  $\Q_{\geqslant 0}^{n}$  can be seen uniquely as an intersection of an  $l$-dimensional linear subspace and  $\Q_{\geqslant 0}^{n}$  for any $l\leqslant n$. Denote by $P_{ih}$ the $(n-1)$-dimensional linear subspace such that $C_{ih}=P_{ih}\bigcap \Q_{\geqslant 0}^{n} $ and by $Q_{i}$ the $(n-1)$-dimensional linear subspace such that $D_{i}=Q_{i}\bigcap \Q_{\geqslant 0}^{n} $ for $i,h\in[1,n]$.

  Let $P^{\prime}=\bigcup\limits_{i,h\in[1,n]}P_{ih}\cup Q_i$. Then, $C'\subseteq P'$.

  Assume there is at most $p(< n)$ linearly independent vectors in $\Q_{\geqslant 0}^{n}\setminus C^{\prime}$. Let $P_{0}$ be the subspace spanned by these $p$ linearly independent vectors. Then $\Q_{\geqslant 0}^{n}\subseteq P_{0}\bigcup C'\subseteq P_0\bigcup P^{\prime}$. But the standard basis $\{e_1,\cdots,e_n\}\subseteq \Q_{\geqslant 0}^{n}$. It follows that $\Q^{n}\subseteq P_0\bigcup P^{\prime}$, which contradicts to the well-known fact that every finite $n$-dimensional linear space can  not be contained in a union of finitely many subspaces with dimensions less than $n$.

  Hence, we can find $n$ linearly independent vectors in $\Q_{\geqslant 0}^{n}\setminus C^{\prime}$, say $v_{1},\cdots,v_{n}\in \Q_{\geqslant 0}^{n}$, whose coordinates satisfy respectively the condition (\ref{m}).

  Now, we can find an $a\in\Z_{+}$ such that $av_{i}\in\Z^n_{\geqslant 0}$. Without loss of generality, we may think for each $av_i=(m^o_{1i},\cdots,m^o_{ni})$ ($i=1,\cdots,n$), the condition (\ref{m}) still is satisfied. Otherwise, the only possibility is that the first condition in (\ref{m}) is not satisfied, then we can always replace $a$ by $ra$ for certain $r\in\Z_{+}$ such that the first condition in (\ref{m}) is satisfied, too.

  In summary, we can obtain  $\Q$-linearly independent vectors $av_i=(m^o_{1i},\cdots,m^o_{ni})\in\Z^n_{\geqslant 0}$ ($i=1,\cdots,n$) satisfying (\ref{m}).

   And as we discussed above, the following equation is satisfied:\\
  \begin{equation*}
    \begin{pmatrix}
       \lambda_{h1}^{i} & \cdots & \lambda_{hn}^{i}
    \end{pmatrix}
    \begin{pmatrix}
       m^o_{11} & \cdots & m^o_{1n} \\
       \vdots & \ddots & \vdots \\
       m^o_{n1} & \cdots & m^o_{nn}
    \end{pmatrix}
     =0
  \end{equation*}
  So $(\lambda_{h1}^{i},\cdots,\lambda_{hn}^{i})$ can only be $(0,\cdots,0)$ for any $h$, $2\leqslant i\leqslant l_{1}$. Then it follows from Lemma \ref{inner} that for any $h$ and   $2\leqslant i\leqslant l_{1}$,
  \begin{equation}\label{key}
   (a_{i1}^{h},\cdots,a_{in}^{h})=(0,\cdots,0).
  \end{equation}
  Then, we have $g(X_{h})=k_{1}^{h}X_{h}+k_{2}^{h}$ where $k_{2}^{h}\in \mathbb{Z}[q^{\pm \frac{1}{2}}]$.

  For general $(m_{1},\cdots,m_{n})\in \mathbb{Z}^n$ such that $X_{1}^{m_{1}}\cdots X_{n}^{m_{n}}$ is a Laurent monomial in $A_q$. According to our above discussion and by (\ref{key}), we have
  \begin{equation*}
     k_{1}^{1}(1-q^{\sum \limits_{t=1}^{n}m_{t}\lambda_{t1}})X_{1}^{m_{1}+1}X_{2}^{m_{2}}\cdots X_{n}^{m_{n}}=\{X_{1},X_{1}^{m_{1}}\cdots X_{n}^{m_{n}}\}=\sum \limits_{j}f_{j}(1-q^{\sum \limits_{t=1}^{n}c_{jt}\lambda_{t1}})X_{1}^{c_{j1}+1}X_{2}^{c_{j2}}\cdots X_{n}^{c_{jn}}
  \end{equation*}
  So we have $\sum \limits_{t=1}^{n}c_{jt}\lambda_{t1}=0$ for any $j\geqslant 2$. Replacing $X_{1}$ by $X_{h}$, $h\in[1,n]$, we obtain that $(c_{j1},\cdots,c_{jn})\Lambda=0$ for any $j\geqslant 2$. However, it contradicts to that  $\Lambda$ is invertible since we have assumed $(c_{j1},\cdots,c_{jn})\not=(0,\cdots,0)$.

  Hence $(c_{j1},\cdots,c_{jn})=(0,\cdots,0)$ for any $j\geqslant 2$.

  Then by (\ref{expen}), we get $g(X_{1}^{m_{1}}\cdots X_{n}^{m_{n}})=f_{1}X_{1}^{m_{1}}\cdots X_{n}^{m_{n}}+f_{2}$, where $f_{1},f_{2}\in \mathbb{Z}[q^{\pm \frac{1}{2}}]$.
  That is, for any cluster Laurent monomial $X$ in $A_q$,
  \begin{equation}\label{gform}
  g(X)=k_{X}X+k_{X}^{\prime},
  \end{equation}
  where $k_{X},k_{X}^{\prime}\in \mathbb{Z}[q^{\pm \frac{1}{2}}]$.\vspace{2mm}

  {\em (ii)}\; For any $g\in \mathscr{P}(A_{q})$, by (\ref{gform}), we define $g^{\prime}$ to be the map  satisfying $g^{\prime}(X)=k_{X}X$ for any cluster Laurent monomial $X\in A_q$ and $g'(a)=0$ for any $a\in \mathbb{Z}[q^{\pm \frac{1}{2}}]$. Trivially, $g'\in \mathscr{P}(A_{q})$. Since Im$(g-g^{\prime})\subseteq Z(A_q)$, we have $g\sim g^{\prime}$.

  By Lemma \ref{I1}, there is a scalar $\mathbb{Z}[q^{\pm\frac{1}{2}}]$-transformation  $g_{0}\in \mathscr{P}(A_{q})$ such that $g_{0}\sim g^{\prime}$. It follows that $g_0\sim g$.
\end{Proof}

Combining Lemma \ref{I1} and Lemma \ref{I2}, we get our main result on inner Poisson structures.
\begin{Theorem}\label{r1}
Let $A_{q}$ be a quantum cluster algebra without coefficients, any inner Poisson structure on $A_{q}$ must be a standard Poison structure.
\end{Theorem}
\begin{Proof}
  According to Theorem \ref{g}, any inner Poisson bracket on $A_{q}$ corresponds to a linear transformation $g\in \mathscr{P}(A_{q})$ up to isomorphism. By Lemma \ref{I1} and Lemma \ref{I2},  we can choose a scalar linear transformation $g'$ in the iso-class of $g$, that is, $g'(W)=k_0W$ for a fixed element $k_0\in \mathbb{Z}[q^{\pm \frac{1}{2}}]$ and for any $W\in A_{q}$. It follows that $ham(W)=k_0[W,-]$ for any $W\in A_{q}$, which means the Poisson structure is standard.
\end{Proof}

\vspace{4mm}

\section{On locally inner Poisson structures}
In \cite{LP}, we define the cluster decomposition of a quantum cluster algebra as following. Let $A_{q,I_{1}}, A_{q,I_{2}}$ be two quantum cluster algebras with initial seeds $(\tilde{X}_{I_{1}},\tilde{B}_{I_{1}},\Lambda_{I_{1}})$ and $(\tilde{X}_{I_{2}},\tilde{B}_{I_{2}},\Lambda_{I_{2}})$ respectively, and let $\Theta$ be an $|I_{1}|\times|I_{2}|$ integer matrix satisfying
\begin{equation}\label{condition formula for cluster decomposition}
  \left\{\begin{array}{cc}
         \tilde{B}_{I_{1}}^{\top}\Theta & =O \\
         \Theta\tilde{B}_{I_{2}} & =O.
       \end{array}
       \right.
\end{equation}
Define $A_{q,I_{1}}\bigsqcup_{\Theta} A_{q,I_{2}}$ to be the algebra equivalent to $A_{q,I_{1}}\bigotimes_{\Z[q^{\pm\frac{1}{2}}]} A_{q,I_{2}}$ as a $\Z[q^{\pm\frac{1}{2}}]$-module
with twist multiplication:
\begin{equation}\label{multi}
(a\otimes b)(c\otimes d)=\sum\limits_{i,j}k_{i}l_{j}q^{\frac{1}{2}\bar{r}_{i}^{\top}\Theta \bar{s}_{j}}a\tilde{X}_{I_{1}}^{\bar{s}_{j}}\otimes\tilde{X}_{I_{2}}^{\bar{r}_{i}}d
\end{equation}
for $b=\sum\limits_{i}k_{i}\tilde{X}_{I_{2}}^{\bar{r}_{i}},c=\sum\limits_{j}l_{j}\tilde{X}_{I_{1}}^{\bar{s}_{j}}$, where $\bar{r}_{i},\bar{s}_{j}$ are exponential column vectors.

Let $A_{q,I_{i}}$ be quantum cluster algebras with initial seeds $(\tilde{X}_{I_{i}},\tilde{B}_{I_{i}},\Lambda_{I_{i}})$ for $i\in [1,r]$. $\bigsqcup$ is associative in the sense that
\begin{center}
  $(A_{q,I_{1}}\bigsqcup_{\Theta_{1}} A_{q,I_{2}})\bigsqcup_{\Theta^{\prime}} A_{q,I_{3}}=A_{q,I_{1}}\bigsqcup_{\Theta^{\prime\prime}} (A_{q,I_{2}}\bigsqcup_{\Theta_{3}}A_{q,I_{3}})$,
\end{center}
where
\[\Theta^{\prime}=\begin{pmatrix}
                        \Theta_{2} \\
                        \Theta_{3}
                      \end{pmatrix},\quad
  \Theta^{\prime\prime}=\begin{pmatrix}
                          \Theta_{1} & \Theta_{2}
                        \end{pmatrix}.\]
\begin{Theorem}\cite{LP}
  Let $A_{q}$ be a quantum cluster algebra with initial seed $(\tilde{X},\tilde{B},\Lambda)$ and $\left\{-,-\right\}$ a compatible Poisson bracket on $A_q$. Assume $\Omega$ is the Poisson matrix of the initial cluster with respect to $\{-,-\}$, $\tilde{B}$ has the decomposition $\tilde{B}=\bigoplus\limits_{i=1}^{r}\tilde{B}_{I_{i}}$ with indecomposables $\tilde{B}_{I_{i}}$ for $i\in[1,r]$, and $A_{q,I_{i}}$ is the quantum cluster indecomposable subalgebra of $A_q$ determined by $(\tilde{B}_{I_{i}},\Lambda_{I_{i}})$. Then $A_{q}\cong\bigsqcup\limits_{i} A_{q,I_{i}}$.
\end{Theorem}

We call $A_q\cong\bigsqcup\limits_{i=1}^r A_{q,I_{i}}$ a {\bf cluster decomposition} of $A_{q}$.

In particular, when $A_{q}$ is a quantum cluster algebra without coefficients, $B$ is invertible. Hence by (\ref{condition formula for cluster decomposition}), we have $\Theta=O$. So, from (\ref{multi}), we also obtain $(a\otimes b)(c\otimes d)=ac\otimes bd$ in this case, which means the cluster decomposition is exactly tensor decomposition $A_{q}=\bigotimes_{i=1}^{r}A_{q,I_{i}}$.

 We generalize inner Poisson structures to locally inner structures in the sense of cluster decomposition.

\begin{Definition}
  Let $A_{q}$ be a quantum cluster algebra with the cluster decomposition $A_{q}=\bigsqcup_{i=1}^{r}A_{q,I_{i}}$.

  (1)A Poisson structure $\left\{-,-\right\}$ on $A_{q}$ is said to be \textbf{locally inner} if for any $a\in A_{q}$ and  $i\in[1,r]$, there is $a_{i}\in A_{q,I_{i}}$ such that $ham(a)|_{A_{q,I_i}}=[a_i,-]$.

  (2)A Poisson structure $\left\{-,-\right\}$ on $A_{q}$ is a \textbf{locally standard Poisson structure} if $\{X_{i},X_{j}\}=0$ when $i$ and $j$ are from different $I_{r}$ and $\left\{-,-\right\}$ is of standard poisson structure on each $X_{I_{r}}$, i.e, $\{X_{i},X_{j}\}=a_{r}[X_{i},X_{j}]$, where $i,j\in I_{r},a_{r}\in \mathbb{Z}[q^{\pm\frac{1}{2}}]$.
\end{Definition}

\begin{Proposition}\label{r2}
  Let $A_{q}$ be a quantum cluster algebra without coefficients, any locally inner Poisson structure on $A_{q}$ is locally standard.
\end{Proposition}
\begin{proof}
  Assume $A_{q}$ has the cluster decomposition $A_{q}=\bigotimes_{i=1}^{r}A_{q,I_{i}}$ and $\left\{-,-\right\}$ is a locally inner Poisson bracket on $A_{q}$. According to the definition, $\left\{-,-\right\}$ is inner when restricted on each $A_{q,I_{i}}$. Hence by Theorem \ref{r1}, $ham(a)\mid_{A_{q,I_{i}}}=\lambda_{i}[a,-]$ for some $\lambda_{i}\in \mathbb{Z}[q^{\pm\frac{1}{2}}]$ and any $a\in A_{q,I_{i}}$. Moreover, for any $a\in A_{q,I_{i}},a^{\prime}\in A_{q,I_{j}}$ and $i\neq j$,
  \[\left\{a,a^{\prime}\right\}=[a_{j},a^{\prime}]\in A_{q,I_{j}},\; \;
  \text{and}\;\;
  \left\{a,a^{\prime}\right\}=[a,a_{i}^{\prime}]\in A_{q,I_{i}}.\]
 So, $\left\{a,a^{\prime}\right\}\in A_{q,I_{i}}\bigcap A_{q,I_{j}}=\Z[q^{\pm\frac{1}{2}}]$ according to the independence of cluster variables and $I_i\cap I_j=\emptyset$. And, note that the expansions of $[a_{j},a^{\prime}]$ and $[a,a_{i}^{\prime}]$ will not contain non-zero constant terms in $\Z[q^{\pm\frac{1}{2}}]$ due to the definitions of the operation $[\;,\;]$ and quantum torus. Thus, $\left\{a,a^{\prime}\right\}=0$.

  Therefore, for any $a\in A_{q,I_{i}}$,
  \begin{equation*}
    ham(a)\mid_{A_{q,I_{j}}}=\left\{
    \begin{array}{cc}
      \lambda_{i}[a,-] & i=j; \\
      0 & i\neq j.
    \end{array}
    \right.
  \end{equation*}
  Then the Poisson structure is exactly locally standard.
\end{proof}

We have the following definition and result in \cite{LP}.
\begin{Definition}
    (1)\; For a quantum cluster algebra $A_{q}$, one of its extended cluster $\tilde X(t)=(X_{1},\cdots,X_{m})$ at $t\in \mathbb{T}_n$ is said to be \textbf{log-canonical} with respect to a Poisson structure $(A_{q},\cdot,\left\{-,-\right\})$ if $\left\{X_{i},X_{j} \right\}=\omega_{ij}X^{e_{i}+e_{j}}$, where $\omega_{ij}\in \mathbb{Z}[q^{\pm \frac{1}{2}}]$ for any $i,j\in[1,m]$.

    (2)\;  A Poisson structure $\left\{-,-\right\}$ on a quantum cluster algebra $A_{q}$ is called \textbf{compatible} with $A_{q}$ if all clusters in $A_{q}$ are log-canonical with respect to $\left\{ -,-\right\}$.
\end{Definition}

\begin{Theorem}\label{compatible}\cite{LP}
  Let $A_{q}$ be a quantum cluster algebra without coefficients. Then a Poisson structure $\left\{-,-\right\}$ on $A_{q}$ is compatible with $A_{q}$ if and only if it is locally standard on $A_{q}$.
\end{Theorem}

Since a locally standard Poisson structure is evidently locally inner, combining Proposition  \ref{r2} and Theorem  \ref{compatible}, we have the final conclusion:
\begin{Theorem}\label{r3}
  Let $A_{q}$ be a quantum cluster algebra without coefficients and $\left\{-,-\right\}$ a Poisson structure on $A_{q}$. The following statements are equivalent:

  (1)\; $\left\{-,-\right\}$ is locally standard.

  (2)\; $\left\{-,-\right\}$ is locally inner.

  (3)\; $\left\{-,-\right\}$ is compatible with $A_{q}$.

\end{Theorem}
\vspace{4mm}
{\bf Acknowledgements:}\; {\em This project is supported by the National Natural Science Foundation of China(No.11671350) and the Zhejiang Provincial Natural Science Foundation of China (No. LY19A010023).}

\end{document}